\documentclass[10pt]{article}
\usepackage{amsmath,amssymb,amsthm}  
\usepackage{fancybox}  
\usepackage{graphicx}
\usepackage{color}
\usepackage{colortbl}
\usepackage{mathdots}
\usepackage{lscape}
\usepackage{arydshln}
\usepackage{ulem} 
\usepackage{xcolor} 
\usepackage{cancel}
\usepackage{tikz}
\usepackage{float}

\usetikzlibrary{decorations.pathreplacing,calc}

\allowdisplaybreaks
 \numberwithin{equation}{section}
 
\begin{document}

\def\fl#1{\left\lfloor#1\right\rfloor}
\def\cl#1{\left\lceil#1\right\rceil}
\def\ang#1{\left\langle#1\right\rangle}
\def\stf#1#2{\left[#1\atop#2\right]}
\def\sts#1#2{\left\{#1\atop#2\right\}}
\def\eul#1#2{\left\langle#1\atop#2\right\rangle}
\def\N{\mathbb N}
\def\Z{\mathbb Z}
\def\R{\mathbb R}
\def\C{\mathbb C}
\def\k{\kappa}
\newcommand{\ctext}[1]{\raise0.2ex\hbox{\textcircled{\scriptsize{#1}}}}

\newtheorem{theorem}{Theorem}
\newtheorem{Prop}{Proposition}
\newtheorem{Cor}{Corollary}
\newtheorem{Lem}{Lemma}

\newenvironment{Rem}{\begin{trivlist} \item[\hskip \labelsep{\it
Remark.}]\setlength{\parindent}{0pt}}{\end{trivlist}}

\title{The Frobenius number for the triple of the $2$-step star numbers
}

\author{
Takao Komatsu
\\
\small Faculty of Education, Nagasaki University\\[-0.8ex]
\small Nagasaki 852-8521 Japan\\[-0.8ex]
\small \texttt{komatsu@nagasaki-u.ac.jp}\\\\
Ritika Goel
\\
\small Department of Mathematics, Shiv Nadar Institute of Eminence\\[-0.8ex]
\small Gautam Buddha Nagar - 201314 India\\[-0.8ex]
\small \texttt{rg796@snu.edu.in}\\\\
Neha Gupta\footnote{Corresponding author}
\\
\small Department of Mathematics, Shiv Nadar Institute of Eminence\\[-0.8ex]
\small Gautam Buddha Nagar - 201314 India\\[-0.8ex]
\small \texttt{neha.gupta@snu.edu.in}
}

\date{
\small MR Subject Classifications: Primary 11D07; Secondary 05A15, 05A17, 05A19, 11B68, 11D04, 11P81
}

\maketitle
 
\begin{abstract}
In this paper, we give closed-form expressions of the Frobenius number for the triple of the $2$-step star numbers $a n(n-2)+1$ for an integer $a\ge 4$. These numbers have been studied from different aspects for some $a$'s. These numbers can also be considered as variations of the well-known star numbers of the form $6 n(n-1)+1$. We also give closed-form expressions of the Sylvester number (genus) for the triple of the $2$-step star numbers.   
\\
{\bf Keywords:} Frobenius problem, Frobenius number, Sylvester number, Ap\' ery set, $2$-step star numbers.         \\
\end{abstract}

\section{Introduction}  

The star numbers, denoted by $S_n$, is given by the formula $6 n(n-1) + 1$.
For an integer $a\ge 4$, define the $2$-step star numbers $\mathfrak S_{a,n}$ by
\begin{equation}
\mathfrak S_{a,n}=a n(n-2)+1\quad(n\ge 2)\,.
\label{def:gstar}
\end{equation}
Star numbers are often known as the centered $12$-gonal numbers or centered dodecagonal numbers. A star number is a centered figurate number that represents a centered hexagram, such as the board on which Chinese checkers is played. The star numbers are used for a new set of vector-valued Teichm\"uller modular forms, defined on the Teichm\"uller space, strictly related to the Mumford forms, which are holomorphic global sections of the vector bundle \cite{MV}.

The $2$-step star numbers have been studied from different aspects.
When $a=4$ in (\ref{def:gstar}), $\mathfrak S_{4,n}$ denotes one fourth the area of these triangles, using six consecutive terms to create the vertices of a triangle at points $(t_{n-2}, t_{n-1})$, $(t_{n}, t_{n+1})$, and $(t_{n+2}, t_{n+3})$, where $t_n$ is the $n$-th triangular number. The first several values are given by  
$$
\{\mathfrak S_{4,n}\}_{n\ge 2}=1, 13, 33, 61, 97, 141, 193, 253, 321, 397, 481, 573, \dots
$$  
(\cite[A082109]{oeis}).
When $a=5$ in (\ref{def:gstar}), $\mathfrak S_{5,n}$ denotes the number of entries required to describe the options and constraints in Don Knuth's formulation of the $n$ nonattacking queens on an $n\times n$ board problem (\cite[A134593]{oeis}). When $a=7$ in (\ref{def:gstar}), the sequence of $\mathfrak S_{7,n}$ is listed in \cite[A131878]{oeis}.

The $p$-numerical semigroup $S_p(A)$ from the set of the positive integers $\{a_1,a_2,\dots,a_k\}$ ($k\ge 2$) is defined as the set of integers whose nonnegative integral linear combinations of given positive integers $a_1,a_2,\dots,a_k$ are expressed in more than $p$ ways \cite{KY24}.  
For some backgrounds on the number of representations, see, e.g., \cite{Bi21,Cayley,Ko03,tr00}.  
For a set of nonnegative integers $\mathbb N_0$, the set $\mathbb N_0\backslash S_p$ is finite if and only if $\gcd(a_1,a_2,\dots,a_k)=1$.  Then there exists the largest integer $g_p(A):=g(S_p)$ in $\mathbb N_0\backslash S_p$, which is called the {\it $p$-Frobenius number}. The cardinality of $\mathbb N_0\backslash S_p$ is called the {\it $p$-genus} (or the $p$-Sylvester number) and is denoted by $n_p(A):=n(S_p)$.  
This kind of concept is a generalization of the famous Diophantine problem of Frobenius (\cite{ADG,RG}) since $p=0$ is the case when the original Frobenius number $g(A)=g_0(A)$ and the genus $n(A)=n_0(A)$.  

When $k=2$, there exists the explicit closed formula of the $p$-Frobenius number for any non-negative integer $p$. However, for $k=3$, the $p$-Frobenius number cannot be given by any set of closed formulas, which can be reduced to a finite set of certain polynomials (\cite{cu90}). Since it is very difficult to give a closed explicit formula for any general sequence for three or more variables, many researchers have tried to find the Frobenius number for special cases (see, e.g., \cite{RBT18} and references therein).  
Though it is even more difficult when $p>0$ (see, e.g., \cite{Ko22b,Ko23b,KM,KP}),    
in \cite{Ko22a}, the $p$-Frobenius numbers of the consecutive three triangular numbers are studied.
In \cite{LX}, a fast algorithm to calculate the number of representations is given.  

In this paper, the closed-form expressions of the Frobenius numbers of the consecutive three $2$-step numbers $\{\mathfrak S_{a,n},\mathfrak S_{a,n+1},\mathfrak S_{a,n+2}\}$ ($a\ge 5$ and $n\ge 3$) are shown. We also give the explicit formula for their Sylvester number.

\section{Basic properties of the $2$-step star numbers}  

In this section, we shall show some basic formulas for the $2$-step star numbers.

\begin{Prop}
\begin{align*}
&\sum_{n=2}^\infty\frac{1}{\mathfrak S_{a,n}}=\frac{1}{2(a-1)}-\frac{\pi}{2\sqrt{a(a-1)}}\cot\left(\frac{\sqrt{a(a-1)}}{a}\pi\right)\quad(a\ne 0,1),\\
&\sum_{n=0}^\infty\frac{\mathfrak S_{a,n}}{n!}=e,\quad  
\sum_{n=0}^\infty\frac{(-1)^n\mathfrak S_{a,n}}{n!}=\frac{2 a+1}{e},\\
&\sum_{n=0}^\infty\frac{\mathfrak S_{a,n}}{b^n}=\frac{a b(3-b)}{(b-1)^3}+\frac{b}{b-1}\quad(b>1)\,.
\end{align*}
\end{Prop}
\begin{proof}
The second and third identities are trivial from the definition in (\ref{def:gstar}).  
For the first identity,
%
by referring to \cite[Ch.25]{AZ},
put
$$
g(z):=\frac{1}{a z(z-2)+1}=\frac{1}{a(z-\alpha)(z-\beta)}\,,
$$
where
$$
\alpha=\frac{a+\sqrt{a^2-a}}{a}\quad\hbox{and}\quad \beta=\frac{a-\sqrt{a^2-a}}{a}\,.
$$
For $h(z)=\pi g(z)\cot\pi z$,
\begin{align*}
{\rm Res}_{z=\alpha}h&=\lim_{z\to\alpha}(z-\alpha)\pi g(z)\cot\pi z=\lim_{z\to\alpha}\frac{\pi\cot\pi z}{a(z-\beta)}=\frac{\pi\cot\pi\alpha}{a(\alpha-\beta)}\\
&=\frac{\pi}{2\sqrt{a^2-a}}\cot\left(\frac{a+\sqrt{a^2-a}}{a}\pi\right)\\
&=\frac{\pi}{2\sqrt{a^2-a}}\cot\left(\frac{\sqrt{a^2-a}}{a}\pi\right)\,,\\
{\rm Res}_{z=\beta}h&=\lim_{z\to\beta}(z-\beta)\pi g(z)\cot\pi z=\lim_{z\to\beta}\frac{\pi\cot\pi z}{a(z-\alpha)}=\frac{\pi\cot\pi\beta}{a(\beta-\alpha)}\\
&=\frac{-\pi}{2\sqrt{a^2-a}}\cot\left(\frac{a-\sqrt{a^2-a}}{a}\pi\right)\\
&=\frac{\pi}{2\sqrt{a^2-a}}\cot\left(\frac{\sqrt{a^2-a}}{a}\pi\right)\,.
\end{align*}
Hence, we get
\begin{align*}
\sum_{n=-\infty}^\infty\frac{1}{a n(n-2)+1}&=\sum_{n=-\infty}^\infty g(n)=-\left({\rm Res}_{z=\alpha}h+{\rm Res}_{z=\beta}h\right)\\
&=-\frac{\pi}{\sqrt{a^2-a}}\cot\left(\frac{\sqrt{a^2-a}}{a}\pi\right)\,.
\end{align*}
Since
\begin{align*}
&\sum_{n=-\infty}^\infty\frac{1}{a n(n-2)+1}\\
&=\sum_{n=2}^\infty\frac{1}{a n(n-2)+1}+\sum_{n=0}^\infty\frac{1}{a(-n)(-n-2)+1}-\frac{1}{a-1}\\
&=2\sum_{n=2}^\infty\frac{1}{a n(n-2)+1}-\frac{1}{a-1}\,,
\end{align*}
we obtain the first identity.  
 
For the fourth identity, consider the function
$$
f_b(x):=\sum_{n=0}^\infty\left(\frac{x}{b}\right)^n=\frac{b}{b-x}\quad(b>1)\,.
$$
Since
\begin{align*}
f_b'(x)&=\sum_{n=0}^\infty\frac{n}{b^n}x^{n-1}=\frac{b}{(b-x)^2}\,,\\
f_b''(x)&=\sum_{n=0}^\infty\frac{n(n-1)}{b^n}x^{n-2}=\frac{2 b}{(b-x)^3}\,,
\end{align*}
we have
\begin{align*}
\sum_{n=0}^\infty\frac{n(n-2)}{b^n}x^{n-2}&=\sum_{n=0}^\infty\frac{n(n-1)}{b^n}x^{n-2}-\frac{1}{x}\sum_{n=0}^\infty\frac{n}{b^n}x^{n-1}\\
&=\frac{2 b}{(b-x)^3}-\frac{1}{x}\frac{b}{(b-x)^2}\,.
\end{align*}
Hence,
\begin{align*}
\sum_{n=0}^\infty\frac{\mathfrak S_{a,n}}{b^n}&=a\sum_{n=0}^\infty\frac{n(n-2)}{b^n}+\sum_{n=0}^\infty\frac{1}{b^n}\\
&=a\left(\frac{2 b}{(b-1)^3}-\frac{b}{(b-1)^2}\right)+\frac{b}{b-1}\\
&=\frac{a b(3-b)}{(b-1)^3}+\frac{b}{b-1}
\end{align*}
\end{proof}

\section{Ap\'ery set}  

We introduce the $p$-Ap\'ery set \cite{Apery} to obtain the formulas in this paper.
 
Let $p$ be a nonnegative integer. For a set of positive integers $A=\{a_1,a_2,\dots,a_k\}$ with $\gcd(A)=1$ and $a_1=\min(A)$ we denote by
$$
{\rm Ap}_p(A)={\rm Ap}_p(a_1,a_2,\dots,a_k)=\{m_0^{(p)},m_1^{(p)},\dots,m_{a_1-1}^{(p)}\}\,,
$$
the $p$-Ap\'ery set of $A$, where each positive integer $m_i^{(p)}$ $(0\le i\le a_1-1)$ satisfies the conditions:
$$
{\rm (i)}\, m_i^{(p)}\equiv i\pmod{a_1},\quad{\rm (ii)}\, m_i^{(p)}\in S_p(A),\quad{\rm (iii)}\, m_i^{(p)}-a_1\not\in S_p(A)
$$
Note that $m_0^{(0)}$ is defined to be $0$.

It follows that for each $p$,
$$
{\rm Ap}_p(A)\equiv\{0,1,\dots,a_1-1\}\pmod{a_1}\,.
$$  

One of the convenient formulas to obtain the $p$-Frobenius number is via the elements in the corresponding $p$-Ap\'ery set (\cite{Ko-p}).

\begin{Lem}  
Let $\gcd(a_1,\dots,a_k)=1$ with $a_1=\min\{a_1,\dots,a_k\}$. Then we have
\begin{align}
g_p(a_1,\dots,a_k)&=\max_{0\le j\le a_1-1}m_j^{(p)}-a_1\,,
\label{lem-frob}\\  
n_p(a_1,\dots,a_k)&=\frac{1}{a_1}\sum_{j=0}^{a_1-1}m_j^{(p)}-\frac{a_1-1}{2}\,.
\label{lem-genus}
\end{align}
\label{lem:p-frob}
\end{Lem}  

\noindent
{\it Remark.}  
When $p=0$, the formulas (\ref{lem-frob}) and  (\ref{lem-genus}) are essentially due to Brauer and Shockley \cite{bs62}, and Selmer \cite{se77}, respectively.  
More general formulas, including the $p$-power sum and the $p$-weighted sum, can be seen in \cite{Ko22,Ko-p}.

\section{Main result}

It has been determined that the Frobenius number of the three consecutive $2$-step star numbers can be given as follows.
The general result for a non-negative $p$ is based upon the result for $p=0$, which is stated as follows.    

\begin{theorem}
Let $a$ and $n$ be integers with $a\ge 5$, $n\ge 3$ and not in (\ref{exception-cases}) below.  
When $a\equiv 0\pmod 4$, we have
\begin{align*}
&g_0(\mathfrak S_{a,n},\mathfrak S_{a,n+1},\mathfrak S_{a,n+2})\\
&=(2 n-2)\mathfrak S_{a,n+1}+\frac{a(n-2)+4(n-1)}{4}\mathfrak S_{a,n+2}-\mathfrak S_{a,n}\,.
\end{align*}
When $a\equiv 1\pmod 4$, we have
\begin{align*}
&g_0(\mathfrak S_{a,n},\mathfrak S_{a,n+1},\mathfrak S_{a,n+2})\\
&=\begin{cases}
(4 n-2)\mathfrak S_{a,n+1}+\frac{a(n-2)-2}{4}\mathfrak S_{a,n+2}-\mathfrak S_{a,n}&\text{if $n\equiv 0\pmod 4$};\\
(4 n-1)\mathfrak S_{a,n+1}+\frac{a(n-2)-2 n-1}{4}\mathfrak S_{a,n+2}-\mathfrak S_{a,n}&\text{if $n\equiv 1\pmod 4$};\\
(2 n-2)\mathfrak S_{a,n+1}+\frac{a(n-2)+4 n-4}{4}\mathfrak S_{a,n+2}-\mathfrak S_{a,n}&\text{if $n\equiv 2\pmod 4$}\\
(3 n-2)\mathfrak S_{a,n+1}+\frac{a(n-2)+2 n-3}{4}\mathfrak S_{a,n+2}-\mathfrak S_{a,n}&\text{if $n\equiv 3\pmod 4$}\,.
\end{cases}
\end{align*}
When $a\equiv 2\pmod 4$, we have
\begin{align*}
&g_0(\mathfrak S_{a,n},\mathfrak S_{a,n+1},\mathfrak S_{a,n+2})\\
&=\begin{cases}
(2 n-2)\mathfrak S_{a,n+1}+\frac{a(n-2)+4 n-4}{4}\mathfrak S_{a,n+2}-\mathfrak S_{a,n}&\text{if $n\equiv 0\pmod 2$};\\
(4 n-2)\mathfrak S_{a,n+1}+\frac{a(n-2)-2}{4}\mathfrak S_{a,n+2}-\mathfrak S_{a,n}&\text{if $n\equiv 1\pmod 2$}\,.
\end{cases}
\end{align*}
When $a\equiv 3\pmod 4$, we have
\begin{align*}
&g_0(\mathfrak S_{a,n},\mathfrak S_{a,n+1},\mathfrak S_{a,n+2})\\
&=\begin{cases}
(4 n-2)\mathfrak S_{a,n+1}+\frac{a(n-2)-2}{4}\mathfrak S_{a,n+2}-\mathfrak S_{a,n}&\text{if $n\equiv 0\pmod 4$}\\
(3 n-2)\mathfrak S_{a,n+1}+\frac{a(n-2)+2 n-3}{4}\mathfrak S_{a,n+2}-\mathfrak S_{a,n}&\text{if $n\equiv 1\pmod 4$};\\
(2 n-2)\mathfrak S_{a,n+1}+\frac{a(n-2)+4 n-4}{4}\mathfrak S_{a,n+2}-\mathfrak S_{a,n}&\text{if $n\equiv 2\pmod 4$};\\
(4 n-1)\mathfrak S_{a,n+1}+\frac{a(n-2)-2 n-1}{4}\mathfrak S_{a,n+2}-\mathfrak S_{a,n}&\text{if $n\equiv 3\pmod 4$}\,.
\end{cases}
\end{align*}
\label{th:g-star}
\end{theorem}

\subsection{The construction of the Ap\'ery set}  

We shall use the following structure of the $0$-Ap\'ery set frequently, but it is applicable to many other triplets as well.  Hence, we establish a general setting.    
Consider the set of the triple $A:=\{a_1,a_2,a_3\}$, where $3\le a_1<a_2<a_3$ and $\gcd(a_1,a_2,a_3)=1$.
For simplicity, put $t_{y,z}:=y a_2+z a_3\quad(y,z\ge 0)$. In Table \ref{tb:apery}, denote $(y,z)$ by its position.  

\begin{table}[htbp]

 \caption{${\rm Ap}_0(a_1,a_2,a_3)$}
  \vspace{0.3cm}
 
  \label{tb:apery}
  
  \centering
\scalebox{0.7}{
\begin{tabular}{ccccccc}
\cline{1-2}\cline{3-4}\cline{5-6}\cline{6-7}
\multicolumn{1}{|c}{$(0,0)$}&$\cdots$&$(y_1-1,0)$&$(y_1,0)$&$\cdots$&$\cdots$&\multicolumn{1}{c|}{$(y_0-1,0)$}\\
\multicolumn{1}{|c}{$\vdots$}&&$\vdots$&$\vdots$&&&\multicolumn{1}{c|}{$\vdots$}\\
\multicolumn{1}{|c}{$(0,z_0-1)$}&$\cdots$&$(y_1-1,z_0-1)$&$(y_1,z_0-1)$&$\cdots$&$\cdots$&\multicolumn{1}{c|}{$(y_0-1,z_0-1)$}\\
\cline{4-5}\cline{6-7}
\multicolumn{1}{|c}{$(0,z_0)$}&$\cdots$&\multicolumn{1}{c|}{$(y_1-1,z_0)$}&&&&\\\multicolumn{1}{|c}{$\vdots$}&&\multicolumn{1}{c|}{$\vdots$}&&&&\\
\multicolumn{1}{|c}{$\vdots$}&&\multicolumn{1}{c|}{$\vdots$}&&&&\\
\multicolumn{1}{|c}{$(0,z_1-1)$}&$\cdots$&\multicolumn{1}{c|}{$(y_1-1,z_1-1)$}&&&&\\
\cline{1-2}\cline{3-3}
\end{tabular}
}
  
\end{table}

Here, $(y_0,z_0)$ and $(y_1,z_1)$ satisfy the following conditions.  
\begin{enumerate}
\item[(i)] $0\le y_1\le y_0$ and $0\le z_0\le z_1$
\item[(ii)] $a_1=y_0 z_0+y_1(z_1-z_0)$
\item[(iii)] $y_0 a_2-(z_1-z_0)a_3\equiv0\pmod{a_1}$ and $y_0 a_2>(z_1-z_0)a_3$
\item[(iv)] $y_1 a_2+z_0 a_3\equiv0\pmod{a_1}$
\end{enumerate}
Then, every element of the $0$-Ap\'ery set appears exactly once in the frame of Table \ref{tb:apery}.  
This fact is proved by considering the sequence $\{\ell a_2\pmod{a_1}\}_{\ell=0}^{a_1-1}$ as follows.  

\noindent
{\bf Case 1.}  
Assume that $\gcd(z_0,z_1)=1$ and $\gcd(a_1,a_2)=1$.  
The sequence $\{\ell a_2\pmod{a_1}\}_{\ell=0}^{a_1-1}$ begins from the position $(0,0)$ and takes all elements in that row, and moves to another row by increasing ($z_1-z_0$) columns beginning from $(0,z_1-z_0)$ by using the rule (iii) above. If it is still the longer row of length $y_0$, the sequence takes all elements in that row, and moves to another row by increasing ($z_1-z_0$) columns again.  
If it is the shorter row of length $y_1$, the sequence takes all elements in that row, and moves to another row by decreasing $z_0$ columns by the rule (iv).  
Namely,  
\begin{align*}  
&\{\ell a_2\pmod{a_1}\}_{\ell\ge 0}\\
&=(0,0),(1,0),\dots,(y_0-1,0),\\
&(0,z_1-z_0),(1,z_1-z_0),\dots,(y_0-1,z_1-z_0),\quad\hbox{by (iii)}\\
&\bigl(0,2(z_1-z_0)\bigr),\bigl(1,2(z_1-z_0)\bigr),\dots,\bigl(y_0-1,2(z_1-z_0)\bigr),\quad\hbox{by (iii)}\\
&\dots\\
&\bigl(0,(\cl{\frac{z_0}{z_1-z_0}}-1)(z_1-z_0)\bigr),\bigl(1,(\cl{\frac{z_0}{z_1-z_0}}-1)(z_1-z_0)\bigr),\\
&\qquad \dots,(y_0-1,(\cl{\frac{z_0}{z_1-z_0}}-1)(z_1-z_0)\bigr),\quad\hbox{by (iii)}\\
&\bigl(0,\cl{\frac{z_0}{z_1-z_0}}(z_1-z_0)\bigr),\bigl(1,\cl{\frac{z_0}{z_1-z_0}}(z_1-z_0)\bigr),\\
&\qquad \dots,\bigl(y_1-1,\cl{\frac{z_0}{z_1-z_0}}(z_1-z_0)\bigr),\quad\hbox{by (iii)}\\
&(0,\cl{\frac{z_0}{z_1-z_0}}(z_1-z_0)-z_0\bigr),\dots \quad\hbox{by (iv)}
\end{align*}
Here, for simplicity, we used the position $(y,z)$ instead of $t_{y,z}$.
Thereafter, the sequence repeatedly runs over all the elements of a longer row, then moves to another row by increasing by ($z_1-z_0$) columns, or runs over all the elements of a shorter row then moves to another row by decreasing by $z_0$ columns.
Since $\gcd(z_0,z_1)=1$, all elements of the Ap\'ery set are traced only once without overlap, and finally the end returns to the position $(0,0)$.  

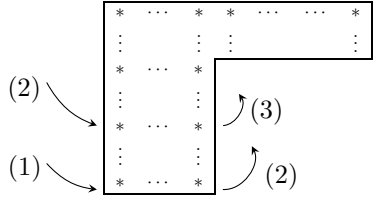
\newcommand{\tikzmark}[2][-3pt]{\tikz[remember picture, overlay, baseline=-0.5ex]\node[#1](#2){};}

\newcounter{arrow}
\setcounter{arrow}{0}
\newcommand{\drawcurvedarrow}[3][]{%
 \refstepcounter{arrow}
 \tikz[remember picture, overlay]\draw (#2.center)edge[#1]node[coordinate,pos=0.1, name=arrow-\thearrow]{}(#3.center);
}

\newcommand{\annote}[3][]{%
 \tikz[remember picture, overlay]\node[#1] at (#2) {#3};
}

\begin{table}[H]
\caption{increasing by condition (iii)}
\vspace{0.5cm}
  \centering
\scalebox{0.7}{
\begin{tabular}{ccccccc}
\cline{1-2}\cline{3-4}\cline{5-6}\cline{6-7}
\multicolumn{1}{|c}{$*$}\tikzmark[xshift= -4em, yshift= -6ex]{o}&$\cdots$&$*$&$*$&$\cdots$&$\cdots$&\multicolumn{1}{c|}{$*$}\tikzmark[xshift= -3.7em, yshift= -4.5ex]{a} \\
\multicolumn{1}{|c}{$\vdots$}\tikzmark[xshift= -2.1em, yshift= -3.8ex]{p}\tikzmark[xshift= -4em, yshift= -4.5ex]{c}&&$\vdots$&$\vdots$&&&\multicolumn{1}{c|}{$\vdots$}\tikzmark[xshift=-2em, yshift=-2ex]{b}\\
\multicolumn{1}{|c}{$*$}\tikzmark[xshift= -2.1em, yshift= -3ex]{d}&$\cdots$&$*$&$*$&$\cdots$&$\cdots$&\multicolumn{1}{c|}{$*$}\tikzmark[xshift= -4em, yshift= -3ex]{e}\\

\multicolumn{1}{|c}{$\vdots$}\tikzmark[xshift= -4em, yshift= -3ex]{q}&&$\vdots$&$\vdots$&&&\multicolumn{1}{c|}{$\vdots$}\tikzmark[xshift= -2em, yshift= -1ex]{f}\\
\multicolumn{1}{|c}{$*$}\tikzmark[xshift= -2.1em, yshift= -1.1ex]{r}&$\cdots$&$*$&$*$&$\cdots$&$\cdots$&\multicolumn{1}{c|}{$*$}\\
\multicolumn{1}{|c}{$\vdots$}&&$\vdots$&$\vdots$&&&\multicolumn{1}{c|}{$\vdots$}\\
\multicolumn{1}{|c}{$\vdots$}\tikzmark[xshift= -4em, yshift= -0.1ex]{i}&&$\vdots$&$\vdots$&&&\multicolumn{1}{c|}{$\vdots$}\\
\multicolumn{1}{|c}{$*$}\tikzmark[xshift= -2.1em, yshift= 2.5ex]{j}&$\cdots$&$*$&$*$&$\cdots$&$\cdots$&\multicolumn{1}{c|}{$*$}\tikzmark[xshift= -4em, yshift= 2.5ex]{g}\\
\cline{4-5}\cline{6-7}

\multicolumn{1}{|c}{$\vdots$}\tikzmark[xshift= -4em, yshift= 2ex]{k}&&\multicolumn{1}{c|}{$\vdots$}\tikzmark[xshift= 5em, yshift= 3ex]{h}\tikzmark[xshift= -1em, yshift= 3.5ex]{n}&&&&\\
\multicolumn{1}{|c}{$*$}\tikzmark[xshift= -2.1em, yshift= 4.5ex]{l}&$\cdots$&\multicolumn{1}{c|}{$*$}\tikzmark[xshift= -1.5em, yshift= 4ex]{m}&&&&\\
\cline{1-2}\cline{3-3}
\end{tabular}
}

  \drawcurvedarrow[bend left=30,-stealth]{a}{b}
    \drawcurvedarrow[bend right=20,-stealth]{c}{d}
    \drawcurvedarrow[bend left=30,-stealth]{e}{f}
    \drawcurvedarrow[bend left=60,-stealth]{g}{h}
     \drawcurvedarrow[bend right=20,-stealth]{i}{j}
     \drawcurvedarrow[bend right=20,-stealth]{k}{l}
     \drawcurvedarrow[bend right=60,-stealth]{m}{n}
     \drawcurvedarrow[bend right=20,-stealth]{o}{p}
     \drawcurvedarrow[bend right=20,-stealth]{q}{r}

   \annote[right,xshift=0.5em,yshift=0.5em]{arrow-1}{(1)}
  \annote[left]{arrow-2}{(1)}
  \annote[right,xshift=1.5em,yshift=-1em]{arrow-3}{(2)}
   \annote[right,xshift=0.5em,yshift=-1.5em]{arrow-4}{(j)}
  \annote[left]{arrow-5}{(j-1)}
  \annote[left]{arrow-6}{(j)}
   \annote[right,xshift=0.5em,yshift=0.5em]{arrow-7}{(j+1)}
  \annote[left,xshift=-0.1em,yshift=0.5em]{arrow-8}{(j+1)}
  \annote[left]{arrow-9}{(2)}
  \label{tb:apery2}
\end{table}

\begin{table}[htbp]
\caption{decreasing by condition (iv)}
\vspace{0.3cm}
  \centering
\scalebox{0.7}{
\begin{tabular}{ccccccc}
\cline{1-2}\cline{3-4}\cline{5-6}\cline{6-7}
\multicolumn{1}{|c}{$*$}&$\cdots$&$*$&$*$&$\cdots$&$\cdots$&\multicolumn{1}{c|}{$*$}\\
\multicolumn{1}{|c}{$\vdots$}&&$\vdots$&$\vdots$&&&\multicolumn{1}{c|}{$\vdots$}\\

\cline{4-5}\cline{6-7}
\multicolumn{1}{|c}{$*$}&$\cdots$&\multicolumn{1}{c|}{$*$}&&&&\\\multicolumn{1}{|c}{$\vdots$}\tikzmark[xshift= -4em, yshift= 2ex]{u}&&\multicolumn{1}{c|}{$\vdots$} \tikzmark[xshift= -1em, yshift= 1ex]{z}\tikzmark[xshift= -0.5em, yshift= -3.6ex]{y}&&&&\\
\multicolumn{1}{|c}{$*$}\tikzmark[xshift= -2.1em, yshift= 1ex]{v}&$\cdots$&\multicolumn{1}{c|}{$*$}\tikzmark[xshift= -1.6em, yshift= 1ex]{w}&&&&\\
\multicolumn{1}{|c}{$\vdots$}\tikzmark[xshift= -4em, yshift= 2ex]{s}&&\multicolumn{1}{c|}{$\vdots$}&&&&\\
\multicolumn{1}{|c}{$*$}\tikzmark[xshift= -2.1em, yshift= 2.5ex]{t}&$\cdots$&\multicolumn{1}{c|}{$*$} \tikzmark[xshift= -1.6em, yshift= 2.5ex]{x}&&&&\\
\cline{1-2}\cline{3-3}
\end{tabular}
}
  
  \drawcurvedarrow[bend right=60,-stealth]{x}{y}
   \drawcurvedarrow[bend right=20,-stealth]{s}{t}
   \drawcurvedarrow[bend right=20,-stealth]{u}{v}
   \drawcurvedarrow[bend right=60,-stealth]{w}{z}

   \annote[left]{arrow-11}{(1)}
   \annote[right,xshift=1em,yshift=0.5em]{arrow-10}{(2)}
   \annote[left]{arrow-12}{(2)}
   \annote[right,xshift=0.5em,yshift=0.5em]{arrow-13}{(3)}

  \label{tb:apery3}
\end{table}

By $\gcd(a_1,a_2)=1$, we have $\{\ell a_2\pmod{a_1}\}_{\ell=0}^{a_1-1}=\{\ell\pmod{a_1}\}_{\ell=0}^{a_1-1}$. Thus, every element of the $0$-Ap\'ery set appears exactly once in the frame of Table \ref{tb:apery}.  
\bigskip

\noindent
{\bf Case 2.}  
Assume that $\gcd(z_0,z_1)\ne 1$ or $\gcd(a_1,a_2)\ne 1$ (including the case where $z_0=z_1$, so $y_0=y_1$).    
In fact, if $\gcd(a_1,a_2)=d>1$, then by (iii) and (iv) with $\gcd(a_3,d)=1$ (note that otherwise, $\gcd(a_1,a_2,a_3)\ne 1$), we have $\gcd(z_1-z_0,d)=d$ and $\gcd(z_0,d)=d$, respectively. Hence, $\gcd(z_0,z_1)=d$. Similarly, if $\gcd(z_0,z_1)=d>1$, then $\gcd(a_1,a_2)=d$.  

In this case, the sequence $\{\ell a_2\pmod{a_1}\}_{\ell=0}^{a_1-1}$ is divided into $d$ non-intersecting subsequences $\{\ell a_2+j a_3\pmod{a_1}\}_{\ell=0}^{a_1/d-1}$ ($j=0,1,\dots,d-1$) with the length of period $a_1/d$. Notice that by $d|a_2$ and $\gcd(a_3,d)=1$, $\ell a_2+j a_3\equiv j\pmod d$ ($j=0,1,\dots,d-1$).  We also have
$$
\frac{a_1}{d}a_2+j a_3=\frac{a_2}{d}a_1+j a_3\equiv j a_3\pmod{a_1}\,.
$$
Each subsequence moves over the elements of the longer row, and the shorter row by increasing and decreasing by the columns, respectively, as shown in Case 1.  
Thus, in this case too, every element of the $0$-Ap\'ery set appears exactly once in the frame of Table \ref{tb:apery}.  Namely, for $j=0,1,\dots,z_1-z_0$  
\begin{align*}  
&(0,j),(1,j),\dots,(y_0-1,j),\\
&(0,z_1-z_0+j),(1,z_1-z_0+j),\dots,(y_0-1,z_1-z_0+j),\quad\hbox{by (iii)}\\
&\bigl(0,2(z_1-z_0)+j\bigr),\bigl(1,2(z_1-z_0)+j\bigr),\dots,\bigl(y_0-1,2(z_1-z_0)+j\bigr),\quad\hbox{by (iii)}\\
&\dots\\
&\bigl(0,(\cl{\frac{z_0}{z_1-z_0}}-1)(z_1-z_0)+j\bigr),\bigl(1,(\cl{\frac{z_0}{z_1-z_0}}-1)(z_1-z_0)+j\bigr),\\
&\qquad \dots,\bigl(y_0-1,(y_0-1,(\cl{\frac{z_0}{z_1-z_0}}-1)(z_1-z_0)+j\bigr),\quad\hbox{by (iii)}\\
&\bigl(0,\cl{\frac{z_0}{z_1-z_0}}(z_1-z_0)+j\bigr),\bigl(1,\cl{\frac{z_0}{z_1-z_0}}(z_1-z_0)+j\bigr),\\
&\qquad \dots,\bigl(y_1-1,\cl{\frac{z_0}{z_1-z_0}}(z_1-z_0)+j\bigr),\quad\hbox{by (iii)}\\
&(0,\cl{\frac{z_0}{z_1-z_0}}(z_1-z_0)-z_0+j\bigr),\dots \quad\hbox{by (iv)}
\end{align*}
\bigskip

From the structure of Table \ref{tb:apery}, the largest element of the $0$-Ap\'ery set is either at $(y_0-1,z_0-1)$ or at $(y_1-1,z_1-1)$. By comparing both values directly, if $(y_0-y_1)a_2-(z_1-z_0)a_3>0$, then the value at $(y_0-1,z_0-1)$ is the largest and
$$
g_0(a_1,a_2,a_3)=(y_0-1)a_2+(z_0-1)a_3-a_1\,.
$$
Otherwise, $(y_1-1,z_1-1)$ is the largest and
$$
g_0(a_1,a_2,a_3)=(y_1-1)a_2+(z_1-1)a_3-a_1\,.
$$

\subsection{Proof of Theorem \ref{th:g-star}}
Since $a_{3}-a_1 = 4an$ and $a_2- a_1 = a(2n-1)$, we derive the relation $4na_2 - (2n-1)a_3 \equiv 0\pmod {a_{1}} $  which corresponds to condition (iii). For the condition (iv), we need to find  $y_1$ and $z_0$ that satisfy $y_{1}a_2 + z_{0}a_3 \equiv 0 \pmod {a_{1}}$. Simplifying, this equation reduces to finding $y_1$ and $z_0$, that satisfy $y_{1}a(2n-1) + z_{0}4an = aa_1$. From here, we observe that $y_1$ is of the form $p-1$, where the possible values of $p$ are $n, 2n, 3n, 4n$. Thus, we obtain different values of $y_1$ and $z_0$ for different $p$.\\\\
For example, take $p=n$.  We get $(n-1)a(2n-1) + z_{0}4an = aa_1$. This gives $z_0 =\frac{a(n-2)-2n +3}{4}$ and $y_1 = n-1$. To ensure $z_0$ is an integer, different cases arise depending on the possible values of $a$ and $n$. We give the complete table for possible values of $a$ and $n$ along with values of $z_0$ and $y_1$ for each possible value of $p$.
\begin{table}[h]
\caption{Different cases for $a$ and $n$ for all possible values of $p$} 
\vspace{0.2cm}
\centering
\begin{tabular}{ | l | l | l | l |}
  \hline
   $\textbf{p}$ & $\mathbf{z_0}$ & $\mathbf{y_1}$ & {\boldmath$(a \text{ mod } 4,n \text{ mod } 4) $}\\ \hline
   $n$ & $\frac{a(n-2)-2n +3}{4}$  & n-1 & $(1,1),(3,3)$\\ \hline
   $2n$ &  $\frac{a(n-2)-4n +4}{4}$ & 2n-1 & $a \equiv 0 \pmod 4, (1,2),(2,0),(2,2),(3,2)$\\ \hline
   $3n$ &  $\frac{a(n-2)-6n +5}{4}$ & 3n-1 & $(1,3),(3,1)$\\ \hline
   $4n$ &  $\frac{a(n-2)-8n +6}{4}$  & 4n-1 & $(1,0), (2,1),(2,3), (3,0) $\\ \hline
 \end{tabular}
\end{table}

We now discuss each case for the possible values of $a$ and calculate $g_0$ in each case.

\subsubsection{The case $a\equiv 0\pmod 4$}  

For $a=4\k$, in order to satisfy all the conditions (i),(ii),(iii) and (iv), we can determine  
$$
(y_0-1,z_0-1)=(4 n-1,(\k-1)n-2\k)=\left(4 n-1,\frac{a(n-2)-4 n}{4}\right)
$$
and
$$  
(y_1-1,z_1-1)=(2n-2,(\k+1)n-2\k-1)=\left(2 n-2,\frac{a(n-2)+4 n}{4}-1\right)\,.
$$
Comparing the two candidates,
$t_{2n-2,(\k+1)n-2\k-1}-t_{4 n-1,(\k-1)n-2\k}=2 a n^2+a-2$.  
Since $n\geq 3$ and $a \geq 5$, the above value is positive. Hence, $t_{2n-2,(\k+1)n-2\k-1}$ is the largest element of the Ap'ery set.
So,  
\begin{align*}
&g_0(\mathfrak S_{a,n},\mathfrak S_{a,n+1},\mathfrak S_{a,n+2})\\
&=(2 n-2)\mathfrak S_{a,n+1}+\frac{a(n-2)+4(n-1)}{4}\mathfrak S_{a,n+2}-\mathfrak S_{a,n}\,.
\end{align*}

\subsubsection{The case $a\equiv 1\pmod 4$}  

For $a=4\k+1$, when $n\equiv 0\pmod 4$, in order to satisfy all the conditions, we can determine  
$$
(y_0-1,z_0-1)=\left(4 n-1,\frac{(4\k-7)n-8\k}{4}\right)=\left(4 n-1,\frac{a(n-2)-2(4 n-1)}{4}\right)
$$
and
$$
(y_1-1,z_1-1)=\left(4 n-2,\frac{(4\k+1)n-8\k-4}{4}\right)=\left(4 n-2,\frac{a(n-2)-2}{4}\right)\,.
$$
Comparing the two candidates,  
$t_{4n-2,\frac{(4\k+1)n-8\k-4}{4}}-t_{4 n-1,\frac{(4\k-7)n-8\k}{4}}=2 a n^3 +2 a n^2-2(a-1)n+a-2$.
Since $n\geq 3$ and $a \geq 5$, the above value is positive. Hence, $t_{4 n-2,\frac{a(n-2)-2}{4}}$ is the largest element of the Ap'ery set.
So,
\begin{align*}
&g_0(\mathfrak S_{a,n},\mathfrak S_{a,n+1},\mathfrak S_{a,n+2})\\
&=(4 n-2)\mathfrak S_{a,n+1}+\frac{a(n-2)-2}{4}\mathfrak S_{a,n+2}-\mathfrak S_{a,n}\,.
\end{align*}

When $n\equiv 1\pmod 4$, we can determine
$$
(y_0-1,z_0-1)=\left(4 n-1,\frac{(4\k-1)n-8\k-3}{4}\right)=\left(4 n-1,\frac{a(n-2)-2 n-1}{4}\right)
$$
and
$$
(y_1-1,z_1-1)=\left(n-2,\frac{(4\k+7)n-8\k-7}{4}\right)=\left(n-2,\frac{a(n-2)+6 n-5}{4}\right)\,.
$$
Comparing the two candidates,  
$t_{4n-1,\frac{(4\k-1)n-8\k-3}{4}}-t_{n-2,\frac{(4\k+7)n-8\k-7}{4}}=a n^3-2 a n^2-(a-1)n-a+2$.  
Since $n\geq 3$ and $a \geq 5$, the above value is positive. Hence, $t_{4n-1,\frac{a(n-2)-2 n-1}{4}}$ is the largest element of the Ap'ery set. So,
\begin{align*}
&g_0(\mathfrak S_{a,n},\mathfrak S_{a,n+1},\mathfrak S_{a,n+2})\\
&=(4 n-1)\mathfrak S_{a,n+1}+\frac{a(n-2)-2 n-1}{4}\mathfrak S_{a,n+2}-\mathfrak S_{a,n}\,.
\end{align*}

When $n\equiv 2\pmod 4$, we can determine
$$
(y_0-1,z_0-1)=\left(4 n-1,\frac{(4\k-3)n-8\k-2}{4}\right)=\left(4 n-1,\frac{a(n-2)-4 n}{4}\right)
$$
and
$$
(y_1-1,z_1-1)=\left(2n-2,\frac{(4\k+5)n-8\k-6}{4}\right)=\left(2n-2,\frac{a(n-2)+4 n-4}{4}\right)\,.
$$
Comparing the two candidates,
$t_{2n-2,\frac{(4\k+5)n-8\k-6}{4}}-t_{4 n-1,\frac{(4\k-3)n-8\k-2}{4}}=2 a n^2+a-2$.
Since $n\geq 3$ and $a \geq 5$, the above value is positive. Hence, $t_{2n-2,\frac{a(n-2)+4 n-4}{4}}$ is the largest element of the Ap'ery set.
So,
\begin{align*}
&g_0(\mathfrak S_{a,n},\mathfrak S_{a,n+1},\mathfrak S_{a,n+2})\\
&=(2 n-2)\mathfrak S_{a,n+1}+\frac{a(n-2)+4 n-4}{4}\mathfrak S_{a,n+2}-\mathfrak S_{a,n}\,.
\end{align*}

When $n\equiv 3\pmod 4$, we can determine
$$
(y_0-1,z_0-1)=\left(4 n-1,\frac{(4\k-5)n-8\k-1}{4}\right)=\left(4 n-1,\frac{a(n-2)-6 n+1}{4}\right)
$$
and
$$
(y_1-1,z_1-1)=\left(3 n-2,\frac{(4\k+3)n-8\k-5}{4}\right)=\left(3 n-2,\frac{a(n-2)+2 n-3}{4}\right)\,.
$$
Comparing the two candidates,
$t_{3 n-2,\frac{(4\k+3)n-8\k-5}{4}}-t_{4 n-1,\frac{(4\k-5)n-8\k-1}{4}}=a n^3+2 a n^2-(a-1)n+a-2$.  
Since $n\geq 3$ and $a \geq 5$, the above value is positive. Hence, $t_{3 n-2,\frac{a(n-2)+2 n-3}{4}}$ is the largest element of the Ap'ery set.
So,
\begin{align*}
&g_0(\mathfrak S_{a,n},\mathfrak S_{a,n+1},\mathfrak S_{a,n+2})\\
&=(3 n-2)\mathfrak S_{a,n+1}+\frac{a(n-2)+2 n-3}{4}\mathfrak S_{a,n+2}-\mathfrak S_{a,n}\,.
\end{align*}

\subsubsection{The case $a\equiv 2\pmod 4$}

For $a=4\k+2$, when $n\equiv 0\pmod 2$, we can determine  
$$
(y_0-1,z_0-1)=\left(4 n-1,\frac{(2\k-1)n-4\k-2}{2}\right)=\left(4 n-1,\frac{a(n-2)-4 n}{4}\right)
$$
and
$$
(y_1-1,z_1-1)=\left(2 n-2,\frac{(2\k+3)n-4\k-4}{2}\right)=\left(2 n-2,\frac{a(n-2)+4 n-4}{4}\right)\,.
$$
Comparing the two candidates,  
$t_{2n-2,\frac{(4\k+5)n-8\k-6}{4}}-t_{4 n-1,\frac{(4\k-3)n-8\k-2}{4}}=2 a n^2+a-2$.  
Since $n\geq 3$ and $a \geq 5$, the above value is positive. Hence, $t_{2n-2,\frac{a(n-2)+4 n-4}{4}}$ is the largest element of the Ap'ery set.  
So,
\begin{align*}
&g_0(\mathfrak S_{a,n},\mathfrak S_{a,n+1},\mathfrak S_{a,n+2})\\
&=(2 n-2)\mathfrak S_{a,n+1}+\frac{a(n-2)+4 n-4}{4}\mathfrak S_{a,n+2}-\mathfrak S_{a,n}\,.
\end{align*}

When $n\equiv 1\pmod 2$, we can determine
$$
(y_0-1,z_0-1)=\left(4 n-1,\frac{(2\k-3)n-4\k-1}{2}\right)=\left(4 n-1,\frac{a(n-2)-8 n+2}{4}\right)
$$
and
$$
(y_1-1,z_1-1)=\left(4 n-2,\frac{(2\k+1)n-4\k-3}{2}\right)=\left(4 n-2,\frac{a(n-2)-2}{4}\right)\,.
$$
Comparing the two candidates,
$t_{4 n-2,\frac{(4\k+1)n-8\k-4}{4}}-t_{4 n-1,\frac{(4\k-7)n-8\k}{4}}=2 a n^3+2 a n^2-2(a-1)n+a-2 $.  
Since $n\geq 3$ and $a \geq 5$, the above value is positive. Hence, $t_{4 n-2,\frac{a(n-2)-2}{4}}$ is the largest element of the Ap'ery set.
So,
\begin{align*}
&g_0(\mathfrak S_{a,n},\mathfrak S_{a,n+1},\mathfrak S_{a,n+2})\\
&=(4 n-2)\mathfrak S_{a,n+1}+\frac{a(n-2)-2}{4}\mathfrak S_{a,n+2}-\mathfrak S_{a,n}\,.
\end{align*}

\subsubsection{The case $a\equiv 3\pmod 4$}  

For $a=4\k+3$, when $n\equiv 0\pmod 4$, we can determine
$$
(y_0-1,z_0-1)=\left(4 n-1,\frac{(4\k-5)n-8\k-4}{4}\right)=\left(4 n-1,\frac{a(n-2)-8 n+2}{4}\right)
$$
and
$$
(y_1-1,z_1-1)=\left(4 n-2,\frac{(4\k+3)n-8\k-8}{4}\right)=\left(4 n-2,\frac{a(n-2)-2}{4}\right)\,.
$$
Comparing the two candidates,  
$t_{4n-2,\frac{(4\k+1)n-8\k-4}{4}}-t_{4 n-1,\frac{(4\k-7)n-8\k}{4}} = 2an^3 +2an^2+2n+a-2an-2 $.
Since $n\geq 3$ and $a \geq 5$, the above value is positive. Hence, $t_{4 n-2,\frac{a(n-2)-2}{4}}$ is the largest element of the Ap'ery set.
 So,
\begin{align*}
&g_0(\mathfrak S_{a,n},\mathfrak S_{a,n+1},\mathfrak S_{a,n+2})\\
&=(4 n-2)\mathfrak S_{a,n+1}+\frac{a(n-2)-2}{4}\mathfrak S_{a,n+2}-\mathfrak S_{a,n}\,.
\end{align*}

When $n\equiv 1\pmod 4$, we can determine
$$
(y_0-1,z_0-1)=\left(4 n-1,\frac{(4\k-3)n-8\k-5}{4}\right)=\left(4 n-1,\frac{a(n-2)-6 n+1}{4}\right)
$$
and
$$
(y_1-1,z_1-1)=\left(3 n-2,\frac{(4\k+5)n-8\k-9}{4}\right)=\left(3 n-2,\frac{a(n-2)+2 n-3}{4}\right)\,.
$$
Comparing the two candidates,
$t_{3 n-2,\frac{(4\k+3)n-8\k-5}{4}}-t_{4 n-1,\frac{(4\k-5)n-8\k-1}{4}}=a n^3+2 a n^2-(a-1)n+a-2$.  
Since $n\geq 3$ and $a \geq 5$, the above value is positive. Hence, $t_{3 n-2,\frac{a(n-2)+2 n-3}{4}}$ is the largest element of the Ap'ery set.
So,
\begin{align*}
&g_0(\mathfrak S_{a,n},\mathfrak S_{a,n+1},\mathfrak S_{a,n+2})\\
&=(3 n-2)\mathfrak S_{a,n+1}+\frac{a(n-2)+2 n-3}{4}\mathfrak S_{a,n+2}-\mathfrak S_{a,n}\,.
\end{align*}

When $n\equiv 2\pmod 4$, we can determine
$$
(y_0-1,z_0-1)=\left(4 n-1,\frac{(4\k-1)n-8\k-6}{4}\right)=\left(4 n-1,\frac{a(n-2)-4 n}{4}\right)
$$
and
$$
(y_1-1,z_1-1)=\left(2 n-2,\frac{(4\k+7)n-8\k-10}{4}\right)=\left(2 n-2,\frac{a(n-2)+4 n-4}{4}\right)\,.
$$
Comparing the two candidates,
$t_{2n-2,\frac{(4\k+5)n-8\k-6}{4}}-t_{4 n-1,\frac{(4\k-3)n-8\k-2}{4}} =2 a n^2+a-2$.
Since $n\geq 3$ and $a \geq 5$, the above value is positive. Hence, $t_{2n-2,\frac{a(n-2)+4 n-4}{4}}$ is the largest element of the Ap'ery set.
So,
\begin{align*}
&g_0(\mathfrak S_{a,n},\mathfrak S_{a,n+1},\mathfrak S_{a,n+2})\\
&=(2 n-2)\mathfrak S_{a,n+1}+\frac{a(n-2)+4 n-4}{4}\mathfrak S_{a,n+2}-\mathfrak S_{a,n}\,.
\end{align*}

When $n\equiv 3\pmod 4$, we can determine
$$
(y_0-1,z_0-1)=\left(4 n-1,\frac{(4\k+1)n-8\k-7}{4}\right)=\left(4 n-1,\frac{a(n-2)-2 n-1}{4}\right)
$$
and
$$
(y_1-1,z_1-1)=\left(n-2,\frac{(4\k+9)n-8\k-11}{4}\right)=\left(n-2,\frac{a(n-2)+6 n-5}{4}\right)\,.
$$
Comparing the two candidates,
$t_{4n-1,\frac{(4\k-1)n-8\k-3}{4}}-t_{n-2,\frac{(4\k+7)n-8\k-7}{4}}=a n^3-2 a n^2-(a-1)n-a+2$.  
Since $n\geq 3$ and $a \geq 5$, the above value is positive. Hence, $t_{4n-1,\frac{a(n-2)-2 n-1}{4}}$ is the largest element of the Ap'ery set.
So,
\begin{align*}
&g_0(\mathfrak S_{a,n},\mathfrak S_{a,n+1},\mathfrak S_{a,n+2})\\
&=(4 n-1)\mathfrak S_{a,n+1}+\frac{a(n-2)-2 n-1}{4}\mathfrak S_{a,n+2}-\mathfrak S_{a,n}\,.
\end{align*}

\subsection{Exceptional cases}\label{subsec:exception}

From the condition (i), $z_0\ge 0$. Hence, the following cases are exceptional, and can be considered one by one independently. 
\begin{align}
&a=5\quad\hbox{and}\quad n\equiv 0,3\pmod 4,\,n=6;\notag\\
&a=6\quad\hbox{and}\quad n\equiv 1,3\pmod 4,\,n=4;\notag\\
&a=7\quad\hbox{and}\quad n\equiv 0\pmod 4,\,n=5,9;\notag\\
&a=9\quad\hbox{and}\quad n=4,8,12;\notag\\
&a=10\quad\hbox{and}\quad n=3,5,7;\notag\\
&a=11\quad\hbox{and}\quad n=4;\notag\\
&a=13\quad\hbox{and}\quad n=4;\notag\\
&a=14\quad\hbox{and}\quad n=3;\notag\\
&a=18\quad\hbox{and}\quad n=3
\,.
\label{exception-cases}
\end{align}

\section{Sylvester number (genus)}  

Observing Table \ref{tb:apery}, the sum of the elements of the $p$-Ap\'ery set is made by
\begin{align*}
&\sum_{j=0}^{a_1-1}m_j^{(0)}
=\sum_{z=0}^{z_0-1}\sum_{y=0}^{y_0-1}t_{y,z}+\sum_{z=z_0}^{z_1-1}\sum_{y=0}^{y_1-1}t_{y,z}\\
&=\frac{1}{2}\bigl((y_0 z_0(a_2(y_0-1)+a_3(z_0-1))-y_1(z_0-z_1)(a_2(y_1-1)+a_3(z_0+z_1-1))\bigr)\\
&=\frac{a_1}{2}\bigl(l a_1-a_2+(z_1-1)a_3+(k-l)y_0 z_0\bigr)\,.
\end{align*}
Here, we used the conditions (ii), (iii) and (iv) as $y_0 a_2-(z_1-z_0)a_3=k a_1$ and $y_1 a_2+z_0 a_3=l a_1$ for some integers $k$ and $l$.
By Lemma \ref{lem:p-frob} (\ref{lem-genus}), we have  
\begin{align}
&n_0(a_1,a_2,a_3)=\frac{1}{a_1}\sum_{j=0}^{a_1-1}m_j^{(0)}-\frac{a_1-1}{2}\notag\\
&=\frac{1}{2}\bigl((l-1)a_1-a_2+(z_1-1)a_3+(k-l)y_0 z_0+1\bigr)\,.
\label{genus-formula}
\end{align}

In the case of $(\mathfrak S_{a,n},\mathfrak S_{a,n+1},\mathfrak S_{a,n+2})$, we apply the condition (iii) to get $y_0=4 n$ and $z_1-z_0=2 n-1$. So, $k=2 n+1$. In addition, from the condition (iv), we get
$$
l=\begin{cases}
\frac{a(n+2)+2 n-1}{4}&\text{if $(a,n)\equiv(1,1),(3,3)\pmod 4$};\\
\frac{a(n+2)+4 n}{4}&\text{if $a\equiv 0$,\,$(a,n)\equiv(1,2),(2,0),(2,2),(3,2)\pmod 4$};\\
\frac{a(n+2)+6 n+1}{4}&\text{if $(a,n)\equiv(1,3),(3,1)\pmod 4$};\\
\frac{a(n+2)+8 n+2}{4}&\text{if $(a,n)\equiv(1,0),(2,1),(2,3),(3,0)\pmod 4$}\,.
\end{cases}
$$
Notice that
$$
l=\frac{a(n+2)+\beta n+\gamma}{4}\Longleftrightarrow z_0-1=\frac{a(n-2)-\beta n+\gamma}{4}
$$
for some integers $\beta$ and $\gamma$.
Hence,
\begin{align*}
&n_0(\mathfrak S_{a,n},\mathfrak S_{a,n+1},\mathfrak S_{a,n+2})\\
&=\frac{1}{2}\bigl((l-1)\mathfrak S_{a,n}-\mathfrak S_{a,n+1}+(z_0+2 n-2)\mathfrak S_{a,n+2}+4 n(2 n-l+1)z_0+1\bigr)\,.
\end{align*}

\begin{theorem}  
Let $a$ and $n$ be integers with $a\ge 5$, $n\ge 3$ and not in (\ref{exception-cases}). Then we have
\begin{align*}
&n_0(\mathfrak S_{a,n},\mathfrak S_{a,n+1},\mathfrak S_{a,n+2})\\
&=\begin{cases}
\frac{1}{8}\bigl((a^2+16 a-12)n^3-(12 a-8)n^2-(4 a^2+14 a-23)n+4 a-10
\bigr)&\\
\qquad\qquad\qquad\text{if $(a,n)\equiv(1,1),(3,3)\pmod 4$};&\\
\frac{1}{8}\bigl((a^2+16 a-16)n^3-12 a n^2-(4 a^2+14 a-24)n+4 a-8
\bigr)&\\
\qquad\qquad\qquad\text{if $a\equiv 0$,\,or\,$(a,n)\equiv(1,2),(2,0),(2,2),(3,2)\pmod 4$};&\\
\frac{1}{8}\bigl((a^2+16 a-12)n^3-(12 a+8)n^2-(4 a^2+14 a-23)n+4 a-6
\bigr)&\\
\qquad\qquad\qquad\text{if $(a,n)\equiv(1,3),(3,1)\pmod 4$};&\\
\frac{1}{8}\bigl((a^2+16 a)n^3-(12 a+16)n^2-(4 a^2+14 a-20)n+4 a-4
\bigr)&\\
\qquad\qquad\qquad\text{if $(a,n)\equiv(1,0),(2,1),(2,3),(3,0)\pmod 4$}\,.&\\
\end{cases}
\end{align*}
\label{th:g-2step-genus}
\end{theorem}

\section{$p$-numerical semigroup}  

As indicated in \cite{Ko22a,Ko22b,Ko-p,KM,KP,Ko23b}, the $p$-Frobenius numbers and the $p$-genus can be obtained by using the $p$-Ap\'ery set, and each element of the $p$-Ap\'ery set can be determined uniquely from the corresponding element of the ($p-1$)-Ap\'ery set. However, the situation becomes more complicated when $p$ becomes larger.  In most cases, detailed discussion is required, depending on the particular case.


\end{document}